\documentclass{amsart}

\usepackage{amstext,amsfonts,amssymb,amscd,amsbsy,amsmath,enumerate}
\usepackage[colorlinks]{hyperref}
\usepackage[all]{xy}
\SelectTips{eu}{}

\newcommand{\seg}[1]{[#1]}

\newcommand{\Gr}{\operatorname{Grass}}
\newcommand{\Ext}{\operatorname{Ext}}

\newcommand{\ext}[4]{\operatorname{Ext}^{#1}_{#2}(#3,#4)}
\newcommand{\gext}[5]{\operatorname{Ext}^{#1}_{#2}(#3,#4){}^{#5}}

\newcommand{\gb}[4]{\operatorname{\beta}_{#1,#2}^{#3}(#4)}

\newcommand{\tor}[4][{}]{\operatorname{Tor}_{#1}^{#2}(#3,#4)}
\newcommand{\gtor}[5]{\operatorname{Tor}_{#1}^{#2}(#3,#4)_{#5}}
\newcommand{\Hom}{\operatorname{Hom}}

\newcommand{\mat}[2]{\operatorname{M}_{#2\times#1}(k)}
\newcommand{\K}{\boldsymbol{K}}
\newcommand{\LL}{\boldsymbol{L}}
\newcommand{\LP}{\boldsymbol{P}}

\newcommand{\Coker}{\operatorname{Coker}}

\newcommand{\rank}{\operatorname{rank}}
\newcommand{\edim}{\operatorname{edim}}

\newcommand{\depth}{\operatorname{depth}}

\newcommand{\shift}{\mathsf\Sigma}

\newcommand{\col}{\colon}

\newcommand{\ges}{\geqslant}

\newcommand{\Ker}{\operatorname{Ker}}

\newcommand{\BZ}{{\mathbb Z}}

\newcommand{\fm}{{\mathfrak m}}

\newcommand{\xra}{\xrightarrow}
\newcommand{\lra}{\longrightarrow}

\newcommand{\bm}{\begin{matrix}}
\newcommand{\dm}{\end{matrix}}

\theoremstyle{plain}
\newtheorem{itheorem}{Theorem}
\newtheorem{theorem}{Theorem}[section]

\newtheorem{proposition}[theorem]{Proposition}
\newtheorem{lemma}[theorem]{Lemma}

\newtheorem{corollary}[theorem]{Corollary}

\theoremstyle{definition}

\newtheorem{chunk}[theorem]{}

\newtheorem{example}[theorem]{Example}

\theoremstyle{remark}

\newtheorem{remark}[theorem]{Remark}

\numberwithin{equation}{theorem}
\numberwithin{subchunk}{theorem}

\newcommand{\indeg}{\operatorname{inf}}
\newcommand{\eps}{\varepsilon}

\newcommand{\dd}{\partial}

\newcommand{\hilb}[2]{H_{#1}(#2)}
\newcommand{\po}[3][R]{P^{#1}_{#2}(#3)}
\newcommand{\Po}[3]{P^{#1}_{#2}(#3)}

\newcommand{\bss}{{\boldsymbol s}}

\begin{document}

\title[Koszul modules]
{Short Koszul modules}

\author[L.~L.~Avramov]{Luchezar L.~Avramov}
\address{Luchezar L.~Avramov\\ Department of Mathematics\\
   University of Nebraska\\ Lincoln\\ NE 68588\\ U.S.A.}
     \email{avramov@math.unl.edu}

\author[S.~B.~Iyengar]{Srikanth B.~Iyengar}
\address{Srikanth B.~Iyengar\\ Department of Mathematics\\
   University of Nebraska\\ Lincoln\\ NE 68588\\ U.S.A.}
     \email{iyengar@math.unl.edu}

\author[L.~M.~\c{S}ega]{Liana M.~\c{S}ega}
\address{Liana M.~\c{S}ega\\ Department of Mathematics and Statistics\\
   University of Missouri\\ \linebreak Kansas City\\ MO 64110\\ U.S.A.}
     \email{segal@umkc.edu}

\dedicatory{To Ralf Fr\"oberg, on his 65th birthday}
\subjclass[2000]{13D02 (primary), 13A02, 13D07 (secondary)}
\keywords{Koszul algebras, Koszul modules}

\thanks{Research partly supported by NSF grants DMS-0803082 (LLA) and
DMS-0903493 (SBI)}

 \begin{abstract}
This article is concerned with graded modules $M$ with linear resolutions
over a standard graded algebra $R$. It is proved that if such an $M$
has Hilbert series $H_M(s)$ of the
form $ps^d+qs^{d+1}$, then the algebra $R$ is Koszul; if, in
addition, $M$ has constant Betti numbers, then $H_R(s)=1+es+(e-1)s^{2}$.
When $H_R(s)=1+es+rs^{2}$ with $r\leq e-1$, and $R$ is Gorenstein or $e=r+1\le3$,
it is proved that generic $R$-modules with $q\leq
(e-1)p$ are linear.
 \end{abstract}

\maketitle

\section*{Introduction}

We study homological properties of graded modules over a standard graded
commutative algebra $R$ over a field $k$; recall that this means that
$R_0$ equals $k$ and $R$ is generated over $k$ by finitely many elements
of degree one.

Unless $R$ is a polynomial ring, any general statement about $R$-modules
necessarily concerns modules of infinite projective dimension.  Various
attractive conjectures have been based on expectations that homological 
properties of modules of finite projective dimension extend---in appropriate 
form---to all modules.

It is remarkable that several such conjectures have been refuted
by using modules $M$, whose infinite minimal free resolution display
the simplest numerical pattern:  the graded Betti numbers $\gb ijRM$
are zero for all $j\ne i$ (that is to say, $M$ is \emph{Koszul}), and
$\gb iiRM=p$ for some $p\ge1$ and all $i\ge0$; see \cite{GP, JS1, JS2}.
Furthermore, in those examples both $R$ and $M$ have special properties:
$R$ is a \emph{Koszul algebra}, meaning that $k$ is a Koszul module,
the Hilbert series $\hilb Rs=\sum_{j\in\mathbb Z}\rank_kR_js^j$ has the
form $1+es+(e-1)s^2$, and one has $\hilb Ms=p+(e-1)ps$.

This is a striking amalgamation of structural and numerical restrictions.
The following result, extracted from Theorems \ref{thm:test} and 
\ref{thm:per}(1), shows that it is inevitable.

\begin{itheorem}
  \label{i1}
Let $R$ be a standard graded algebra and $M$ a non-zero 
$R$-module.

If\,\ $M_j=0$ for $j\ne0,1$ and $M$ is Koszul, then $R$ is a 
Koszul algebra.

If, furthermore, $\gb iiRM=p$ for some $p$ and all $i\ge0$, then
  \[
\hilb Rs=(1+s)\cdot(1+(e-1)s)
  \quad\text{and}\quad
\hilb Ms=p\cdot(1+(e-1)s)\,.
  \]
  \end{itheorem}

The main themes of the paper are to find conditions when such 
modules actually exist, and to establish whether they display
some ``generic'' behavior.  An important step is to identify 
a set-up where similar questions may be stated in meaningful 
terms and answers can be tested against existing examples.  

Much of the discussion is carried out in the broader framework 
of Koszul modules over Koszul algebras.  Conca, Trung, and 
Valla \cite{CTV} proved if $R$ is a Koszul algebra with $\hilb Rs= 
1+es+rs^2$, then $e^{2}\ge4r$ holds, and that \emph{generic} 
quadratic algebras $R$ satisfying this inequality are Koszul.
 
To analyze the restrictions imposed on $M$ by Theorem~\ref{i1},
we fix a Koszul algebra $R$ with $\hilb Rs=1+es+rs^2$ and use 
\emph{multiplication tables} to parametrize the $R$-modules
with underlying vector space $k^p\oplus k^q(-1)$; see Section
\ref{sec:parameters}.  This identifies such modules with the 
points of the affine space $\mat q{ep}$ of $ep\times q$ matrices 
with elements in $k$, equipped with the Zariski topology.

We study the following questions concerning the subset
 $\LL_{p,q}(R)\subseteq\mat q{ep}$ corresponding to Koszul
$R$-modules:  \emph{When is $\LL_{p,q}(R)$ non-empty?  When 
is its interior non-empty?} Recall that, in a topological space, the
\emph{interior} of a subset $X$ is the largest open set contained 
in $X$; in $\mat q{ep}$ every subset with non-empty interior is 
\emph{dense}, because affine spaces are irreducible.

It is not hard to show that $2q\le(e+\sqrt{e^2-4r})p$
is a necessary condition for $\LL_{p,q}$ to  be non-empty; see
Corollary~\ref{n-injective}.   To establish sufficient
conditions, we assume that $e\ge r+1$ holds.  Conca \cite{Co} proved
that, generically, algebras $R$ satisfying this inequality contain
an element $x\in R_1$ with $x^2=0$ and $xR_1=R_2$.  In an earlier
paper,~\cite{AIS}, we called such an $x$ a \emph{Conca generator}
of\,\ $R$ and demonstrated that the existence of one impacts
the structure of the minimal free resolutions of every $R$-module.  
The results of \cite{AIS} are widely used here.

The following statement is condensed from Propositions \ref{prop:(e-1)p},
\ref{prop:e-1}, and \ref{prop:(e-r)p}.  Its proof depends on the study,
in Section~\ref{sec:open}, of the loci $\LL_{p,q}^m(R)$ of modules whose
minimal free resolution is linear for the first $m$ steps.

\begin{itheorem}
\label{i2}
Let $R$ be a standard graded algebra with a Conca generator.

For $p,q\in\mathbb{N}$ the linear locus $\LL_{p,q}(R)$ of\,\ $\mat
q{ep}$ is not empty when $q\le (e-1)p$, and has a non-empty interior
when $q\le\max\{e-1\,,\,(e-r)p\}$.
 \end{itheorem}

In the motivating case when $\hilb Rs=1+es+(e-1)s^2$, Theorem \ref{i2}
shows that \emph{generically} $\LL_{p,(e-1)p}(R)$ is not empty.  Computer
experiments suggest that even its interior may be non-empty.  Indeed,
letting $R$ be a quotient of $k[x_1,\dots,x_e]$ by $\binom{e+1}2-(e-1)$
``random'' quadratic forms and $M$ an $R$-module presented by a ``random"
$p\times p$ matrix of linear forms in $x_1,\dots,x_e$\,, one gets $\gb
ijR{M}=0$ for $j\ne i$ and $\gb iiR{M}=p$ with unsettling frequency and
for ``large'' values of\,\ $i$.

In the next theorem, contained in Propositions \ref{prop:gor} and
\ref{prop:per3}, we describe algebras with non-empty open sets of linear
modules, under mild hypotheses on $k$.

\begin{itheorem}
\label{i3}
Let $R$ be a short standard graded $k$-algebra.

If $R$ is Gorenstein, then for all pairs $(p,q)$ with $p\ge1$ the set 
$\LL_{p,q}(R)$ is open in $\mat q{ep}$; it is not empty when $q\leq (e-1)p$
and there exists a non-zero element $x\in R_1$ with $x^2=0$ (in particular, when 
$k$ is algebraically closed).

If $R$ is quadratic with $\hilb Rs=1+es+(e-1)s^2$ and $e\leq 3$, then for all
$p\ge1$ the set $\LL_{p,(e-1)p}(R)$ is open in $\mat{(e-1)p}{ep}$, and is
not empty if $k$ is infinite.
  \end{itheorem}

For $R$ as in the last statement of Theorem \ref{i3}, the $R$-modules in
$\LL_{p,(e-1)p}(R)$ are described as those that are \emph{periodic of
period $2$}, see Section \ref{sec:per}.  Over rings with $e\ge4$, these 
classes may be distinct, and new ones appear; see \cite{GP, JS2}.

The generic behavior of Koszul modules with constant Betti numbers over a
generic Koszul algebra $R$ with $\hilb Rs=1+4s+3s^2$ still is a mystery.

\section*{Notation}
  \label{sec:notation}

Let \emph{$(R,\fm,k)$ be a graded algebra}; in this paper, the phrase
introduces the following hypotheses and notation: $k$ is a field,
$R=\oplus_{j\in\BZ}R_j$ is a commutative graded $k$-algebra 
finitely generated over $R_0=k$, $R_j=0$ for $j<0$, and 
$\fm=\oplus_{j\ges1}R_j$. 

Let $M=\oplus_{j\in\BZ} M_j$ be a graded $R$-module, here 
always assumed finite. For every $d\in\BZ$, we let $M(d)$ 
denote the graded $R$-modules $M(d)_j=M_{j+d}$ for each $j$.  Set:
  \begin{align*}
\indeg M&=\inf\{j\in\BZ\mid M_j\ne0\}\,,
  \\
\hilb Ms&=\sum _{j=\indeg M}^\infty(\rank_k M_j)s^j\in\BZ(\!(s)\!)\,.
  \end{align*}
The formal Laurent series above is the \emph{Hilbert series} of\,\ $M$.

It is implicitly assumed that homomorphisms of graded $R$-modules
preserve degrees.  In this category, the free modules are isomorphic
to direct sums of copies of\,\ $R(d)$, with various $d$.  Every graded
$R$-module $M$ has a \emph{minimal free resolution}
$$
F=\quad\cdots \to F_n\xra{\dd_n} F_{n-1}\to \cdots\to F_1\xra{\dd_1} F_0\to 0  
$$
with each $F_n$ finite free and $\dd_n(F_n)\subseteq \fm F_{n-1}$.  
Computing with it, one gets 
  \[
\ext iRMk=\bigoplus_{j\in\BZ}\gext iRMkj=\Hom_k((F_i/\fm F_i)_j,k)
  \quad\text{for each}\quad i\ge0\,.
  \]

Composition products turn 
$\mathcal E=\bigoplus_{i\ges0,j\ges0}\Ext^i_R(k,k)^j$ 
into a bigraded $k$-algebra, and 
$\mathcal M=\bigoplus_{i\ges0,j\in\BZ}\Ext^i_R(M,k)^j$ 
into a bigraded $\mathcal E$-module.

The $(i,j)$th {\it graded Betti number} of\,\ $M$ is defined to be
  \[
\gb ijRM=\rank_k\gext iRMkj\,.
  \] 
The \emph{graded Poincar\'e series} of\,\ $M$ over $R$ is the formal
power series
   \begin{equation*}
  \label{2-Poincare}
\po M{s,t}=\sum_{i\in\mathbb N,j\in\BZ}\gb ijRM\,s^jt^i \in\BZ[s^{\pm1}][\![t]\!]\,.
  \end{equation*}
We also use non-graded versions of these notions, namely
  \[
\beta_i^R(M)=\sum_{j\in\BZ}\gb ijRM
  \quad\text{and}\quad
\po M{t}=\sum_{i\in\mathbb N}\beta_i^R(M)t^i=\po M{1,t}\in\BZ[\![t]\!]\,.
  \]

\section{Short linear modules and Koszul algebras}

In this section $(R,\fm,k)$ is a graded algebra.  We recall the 
definitions of the algebras and modules of principal interest
for this paper; see \cite{Fr0} or \cite{PP} for details.

  \begin{chunk}
    \label{ch:linear}
We say that an $R$-module $M$ is {\it linear} if it is graded and $\gb
ijRM=0$ holds for all $j-i\ne d$ and some $d\in\BZ$; in case $M\ne0$ one
has $d=\indeg M$, and $M$ is generated in degree~$d$.  It is well-known
that $M$ is linear if and only it satisfies
  \begin{equation}
 \label{eq:linearg}
\po{M}{s,t}\cdot\hilb{R}{-st}=(-t)^{-d}\hilb{ M}{-st}\,.
  \end{equation}

A linear module $M$ with $\indeg M=0$ is also called a \emph{Koszul
module}.
  \end{chunk}

  \begin{chunk}
    \label{ch:koszul}
The algebra $R$ is {\it Koszul} if\,\ $k$ is a linear $R$-module; the
equalities $\gb 0jR{\fm}=\gb 1jRk$ show that then $R$ is \emph{standard};
that is, it is generated over $k$ by elements of degree $1$. It is
well-known, see \cite[1.16]{BF}, that $R$ is Koszul if and only if
it satisfies
  \begin{equation}
 \label{eq:koszul}
\po{k}{t}\cdot\hilb{R}{-t}=1\,,
  \end{equation}
if and only if the $k$-algebra $\mathcal E$ is generated by $\mathcal
E^{1,1}$; see \cite[Thm.\,1.2]{Lo} or \cite[Ch.\,2,\,\S1]{PP}.
  \end{chunk}

We frequently refer to the following criterion:

  \begin{chunk}
    \label{ch:square}
If $Q$ is a standard graded $k$-algebra and $g$ is a 
non-zero-divisor in $Q_1$ or $Q_2$, then $Q$ and $Q/(g)$ are Koszul 
simultaneously, see \cite[Thm.\,4(e)(iv)]{BF} or \cite[6.3]{PP}.
  \end{chunk}

To link linearity of\,\ $M$ to linearity of\,\ $k$, we recall a
construction.

\begin{chunk}
\label{trivial-extension} 
Let $M\ne0$ be a graded $R$-module, and set $d=\indeg M$.
The \emph{trivial extension} $R\ltimes M$ has $R\oplus M(d-1)$ 
as graded $k$-spaces, and
  \[
(r_1,m_1)\cdot (r_2,m_2)=(r_1r_2,r_1m_2+r_2m_1)
\quad\text{for all $r_j\in R$ and
$m_j\in M(d-1)$}\,.
  \]
Setting $\fm\ltimes M=\fm\oplus M(d-1)$, we get  
a graded $k$-algebra $(R\ltimes M,\fm\ltimes M,k)$.

One has $R\ltimes M=R\ltimes(M(n))$ for every
$n\in\BZ$, and the following equality holds:
\begin{equation}
  \label{trivial-hilbert}
  \hilb {R\ltimes M}s=\hilb{R}{s}+s^{1-d}\hilb{M}{s}\,.
\end{equation}

The graded version of a result of Gulliksen, see \cite[Thm.\,2]{Gu}, reads
\begin{equation}
  \label{trivial-poincare-bigraded}
\Po{R\ltimes M}k{s,t}=\frac{\po k{s,t}}{1-s^{1-d}t\po M{s,t}}
\end{equation}
 \end{chunk}

The implication (i)$\implies$(ii) in the next proposition is obtained in 
\cite[Ch.\,2, 5.5]{PP} by a different argument, which works also in a
non-commutative situation.

  \begin{proposition}
   \label{koszul-rings-modules}
Let $(R,\fm,k)$ be a graded algebra and $M$ a graded $R$-module. 

The following statements then are equivalent:

\begin{enumerate}[\rm\quad(i)]
\item $R$ is a Koszul algebra and $M$ is linear.
\item $R\ltimes M$ is a Koszul algebra. 
\item $R$ is a Koszul algebra, and for some $d\in\BZ$ one has
  \[
\po{M}{t}\cdot\hilb{R}{-t}=(-t)^{-d}\hilb{ M}{-t} \,.
   \]
\end{enumerate}
\end{proposition}

\begin{proof}
(i)$\implies$(iii).  The desired equality is obtained from \eqref{eq:linearg} 
by setting $s=1$.

(iii)$\implies$(ii).
Comparing the orders of the formal Laurent series in (iii), one gets 
$d=\indeg M$.  In the following string of equalities, the first one comes 
from setting $s=1$ in \eqref{trivial-poincare-bigraded}, the second 
from the hypothesis, the last one from  \eqref{trivial-hilbert}. 
\begin{align*}
\po[R\ltimes M]kt&={\po kt}\cdot\frac1{1-t\po Mt}\\
&=\frac{1}{\hilb R{-t}}\cdot \frac{1}{1-t\cdot(-t)^{-d}\hilb{M}{-t}{\hilb{R}{-t}}^{-1}}\\
&=\frac{1}{\hilb{R}{-t}+(-t)^{1-d}\hilb{M}{-t}}\\
&=\frac{1}{\hilb{R\ltimes M}{-t}}\,.
\end{align*}
By \eqref{eq:koszul}, the composite equality implies that $R\ltimes M$ 
is Koszul. 

(ii)$\implies$(i).
The evident homomorphisms $R\to R\ltimes M\to R$  of graded algebras
compose to the identity.  As $\gext i?kkj$ is a functor of the ring argument, 
$\gext iRkkj$ is a direct summand of\,\ $\gext i{R\ltimes M}kkj$.  
Thus, $R$ is Koszul, so both $\po k{s,t}$ and $\po[R\ltimes M]k{s,t}$
can be written as formal power series in $st$.  The equality
$$
\po M{s,t}=
s^d\cdot\frac1{st}\bigg(1-\frac{\po k{s,t}}{\po [R\ltimes M]k{s,t}}\bigg)\,,
$$
which comes from \eqref{trivial-poincare-bigraded}, gives $\gb ijRM=0$
for $j-i\ne d$; thus, $M$ is linear.  \end{proof}

A graded $R$-module $M$ is \emph{short} if\,\ $\hilb M{s}=(p+qs)s^d$
for some $d\in\BZ$.  Koszul algebras have short linear modules,
for $k$ is one, by definition.  Conversely:

\begin{theorem}
\label{thm:test}
Let $R$ be a standard graded algebra.

If\,\ $R$ has a linear module $M\ne0$ that is short, then $R$ is Koszul.
 \end{theorem}

\begin{proof}
Let $\eps\col R/\fm^{2}\to k$ denote the canonical surjection.  Roos
\cite[Cor.\,1, p.\,291]{Ro} proves that $\Ker(\Ext_{R}^{*}(\eps,k))$
is the subalgebra of\,\ $\mathcal{E}=\Ext_{R}^{*}(k,k)$ generated by
$\Ext^1_R(k,k)$, so it suffices to prove that $\gext iR{\eps}kj=0$
holds for all $i,j$; see \ref{ch:koszul}.

Replacing $M$ with $M(d)$, we may assume that $M$ is Koszul.
In an exact sequence 
  \[
0\to N\to R^b\to M\to 0
  \]
of graded modules with $\fm N\subseteq R^{b}$, one has 
$\fm N=N_{\geqslant 2}=R_{\geqslant 2}^n=\fm^2R^n$ because 
$N_1$ generates $N$ and $M_2=0$. Thus, we get a commutative 
diagram with exact rows
\[
\xymatrixrowsep{2.5pc} 
\xymatrixcolsep{2pc} 
\xymatrix{
0\ar@{->}[r] &\fm N \ar@{->}[r]\ar@{->}[d] & R^{b} \ar@{->}[r]\ar@{=}[d] 
           & (R/\fm^{2})^{b} \ar@{->}[r]\ar@{->}[d] & 0 \\
0\ar@{->}[r] & N \ar@{->}[r] & R^{b} \ar@{->}[r] & M \ar@{->}[r] & 0 
}
\]
It induces the square in the following commutative diagram, where the
factorization $(R/\fm^2)^{b}\to M\to k^{b}$
of\,\ $\eps^{b}\col(R/\fm^2)^{b}\to k^{b}$ induces the triangle:
  \[
\xymatrixrowsep{1pc} 
\xymatrixcolsep{2pc} 
\xymatrix{
	&\gext i{R}{(R/\fm^{2})^b}kj \ar@{<-}[dd] \ar@{<-}[r]^-{\cong} 
	               & \gext{i-1}{R}{\fm N}kj\ar@{<-}[dd]^{0} \\
\gext i{R}{k^b}kj\ar@/^1.1pc/[ur]^-{\gext i{R}{\eps^b}kj}	\ar@/_1.1pc/[dr] &&\\
	& \gext {i}{R}Mkj \ar@{<-}[r]^-{\cong} & \gext {i-1}{R}Nkj }
  \]

One has $\Ext_{R}^{i-1}(N,k)^j=0$ for $j\ne i$ by isomorphism in the bottom 
row, and $\Ext^{i-1}_{R}(\fm N,k)^{i}=0$ because $\indeg(\fm N) =2$, so the 
vertical map on the right is zero.  Now the diagram implies 
$\Ext^{i}_{R}(\eps^{b},k)=0$, whence $\Ext^{i}_{R}(\eps,k)=0$.
  \end{proof}

We say that an algebra $R$ is \emph{short} when $R_i=0$ for $i\ge3$.

Existence of linear modules imposes numerical constraints on short algebras.

\begin{corollary}
  \label{n-injective}
Let $R$ be a standard graded with $\hilb Rs=1+es+rs^2$, and set
\[
u={(e-\sqrt{e^2-4r})/2}\quad\text{and}\quad v={(e+\sqrt{e^2-4r})/2}\,.
\]

If there exists a non-free linear $R$-module, then $R$ is Koszul and
$e^{2}\ge4r$ holds.

If $M$ is a linear $R$-module with $\hilb Ms=ps^d+qs^{d+1}$ and $p\ne0$,
then $q\le vp$.

If, furthermore, $q=vp$, then $u$ and $v$ are integers, and there is an equality
  \[
\po M{s,t}=\frac{ps^d}{1-ust}\,.
  \]
  \end{corollary}

 \begin{proof}
Let $N$ be a non-free linear $R$-module, and set $j=\inf(N)$.  One then
has $0\ne\Omega^R_1(N)\subseteq\fm R^{n}(-j)$, so $N$ is short, non-zero,
and linear.  Theorem~\ref{thm:test} shows that $R$ is Koszul, and then
$e^{2}\ge4r$ holds by Conca, Trung, and Valla \cite[3.4]{CTV}.

When $\hilb Ms=ps^d+qs^{d+1}$, from formula \eqref{eq:linearg} we get
  \[
\po M{s,t}=\frac{s^d(p-qst)}{1-est+rst^2}=\frac{s^d(p-qst)}{(1-ust)(1-vst)}\,.
  \]
Prime fraction decomposition yields real numbers $a$ and $b$, such that
  \[
\po M{s,t}=
\begin{cases}
s^d\sum_{i=0}^\infty(au^{i}+bv^i)(st)^{i}
&\text{when }e^2>4r\,,
  \\
s^d\sum_{i=0}^\infty(av^i + b(i+1)v^i)(st)^{i}
&\text{when }e^2=4r\,.
  \end{cases}
  \]
As $M$ is not free, the coefficient of $s^d(st)^i$ is positive for each $i\ge0$,
hence $b\ge0$.

When $e^2>4r$ one has $a+b=p$ and $av+bu=q$, and hence 
$pv=q+b(v-u)\ge q$. When $e^2=4r$ holds, $a+b=p$ and 
$av=q$ give $pv=q+bv\ge q$.

Assume $pv=q$. In both cases one then has $b=0$, which means 
$\po M{s,t}=ps^d(1-ust)^{-1}$.
This implies that $u$ is an integer, and hence so is $v=e-u$.
 \end{proof}

We illustrate the tightness of the hypotheses in the last two results:

\begin{example} 
\label{rem:nonkoszul}
Let $k$ be a field and set $R=k[x,y]/(x^{2},y^{3})$.

The algebra $R$ satisfies $R_{i}=0$ for $i\ge 4$, and the $R$-module
$M=R/xR$ is non-free, linear, with $M_{n}=0$ for $n\ne 0,1,2$.
However, $R$ is not Koszul.  
  \end{example}

  \begin{remark} 
\label{rem:CRV}
If $V$ is a \emph{generic} $k$-subspace of codimension 
$r$ in the space of quadrics in $k[x_1,\dots,x_e]$.  By \cite[3.1]{CTV},
$e^{2}\ge4r$ implies that $k[x_1,\dots,x_e]/(V)$ is short and Koszul; the
converse also holds, due to Fr\"oberg and L\"ofwall \cite[7.1]{FL}.
  \end{remark}
  
Partial versions of the Koszul property are also of interest.

We say that an $R$-module $M$ is \emph{$m$-step linear} for some 
integer $m\ge0$ if it satisfies $\gb i{j}{R}{M}=0$ for all $j\ne i+\inf M$
with $i\le m$.  Thus, $0$-step linear means that $M$ is generated in
a single degree, and $1$-step linear means that, in addition, it has a
free presentation with a presenting matrix of linear forms.

  \begin{proposition}
 \label{lem:m-step}
Let $R$ be a Koszul algebra and $M$ an $R$-module with $\indeg M=0$.

For every non-negative integer $m$ the following conditions are equivalent.
 \begin{enumerate}[\quad\rm(i)]
  \item
$M$ is $m$-step linear.
  \item
$\po{M}{s,t}\equiv \hilb{M}{-st}\cdot\hilb{R}{-st}^{-1} \pmod{t^{m+1}}$
 \end{enumerate}
When $R$ is short, they are also equivalent to
 \begin{enumerate}[\quad\rm(i)]
  \item[\rm(iii)]
$\gb m{m+1}{R}{M}=0$
 \end{enumerate}
  \end{proposition}

  \begin{proof}
(i)$\implies$(ii).  By hypothesis, there is an exact sequence 
  \[
0\to L\to R(-m)^{b_m}\to\cdots\to R(-1)^{b_1}\to R^{b_0}\to M\to 0
  \]
of\,\ $R$-modules with $b_i=\gb i{i}{R}{M}$ for $i\le m$, and $L_j=0$
for $j\le m$.  It yields
  \[
\hilb Ms=\sum_{i=0}^m(-1)^{i}b_is^i\hilb Rs+(-1)^{m+1}\hilb Ls\,.
  \]
Dividing this equality by $\hilb Rs$ and replacing $s$ with $-st$, 
one gets (ii).

(ii)$\implies$(i)$\implies$(iii).  These implications hold by definition.  

(iii)$\implies$(ii).  Let $\mathcal E$ be the bigraded algebra 
$\Ext_{R}^{*}(k,k)$ and $\mathcal M$ its bigraded module 
$\Ext_{R}^{*}(M,k)$.  Write $(M_0)^*$ for the bigraded $k$-vector
space with $\Hom_k(M_0,k)$ in bidegree $(0,0)$ and $0$ elsewhere,
and $(M_1)^*$ for that with $\Hom_k(M_1,k)$ in bidegree $(1,1)$ and $0$
elsewhere.  Graded versions of \cite[2.4, 2.5]{AIS} yield
an exact sequence
  \[
0\lra
\mathcal F\lra
\mathcal E\otimes_k (M_1)^* \lra
(\mathcal E\otimes_k (M_0)^*)\oplus\shift \mathcal F\lra
\mathcal M\lra 0
  \]
of bigraded $\mathcal E$-modules, where $\mathcal F$ is free, 
and $(\shift\mathcal F)^{i,j}=\mathcal F^{i+1,j}$.

One has $\mathcal E^{i,j}=0$ for $i\ne j$, because 
$R$ is Koszul, and $\mathcal F^{i,j}=0$ for $i\ne j$, because of the 
inclusion $\mathcal F\subseteq\mathcal E\otimes_k(M_1)^*$.  Set 
$r_l=\rank_k\mathcal F^{l,l}$.  The exact sequence gives
\[
\po{M}{s,t}=\hilb M{-st}\cdot\po{k}{s,t}
+\left(1+\frac1t\right)\cdot\sum_{l=0}^\infty r_l \,(st)^l
\,.
 \]
Since $R$ is Koszul, \eqref{eq:linearg} gives $\po{k}{s,t}=H_R(-st)^{-1}$, 
so the formula above yields
\[
\po{M}{s,t}-\hilb M{-st}\cdot\hilb R{-st}^{-1}
=\sum_{l=0}^\infty(r_l+sr_{l+1})\,(st)^{l}\,.
 \]
As $\mathcal E^i\ne 0$ for $i\ge 0$ and $\mathcal F$ is a free 
$\mathcal E$-module, $r_l=0$ means $r_i=0$ for $i\le l$. 
  \end{proof}

\section{Parametrizing short modules}
  \label{sec:parameters}

In this section $R$ is a standard graded algebra with $R_1\ne0$ and
$x_1,\dots,x_e$ a fixed $k$-basis of\,\ $R_1$.  Let $p$ be a
positive integer and $q$ a non-negative one.
The goal here is to describe a convenient parameter space for
modules with Hilbert series $p+qs$.  

  \begin{chunk}
    \label{ch:spaces}
Let $\{u_n\}_{n\in{\mathbb N}}$ and $\{v_h\}_{h\in{\mathbb N}}$
be the standard bases of the vector spaces $k^{({\mathbb N})}$
and $k^{({\mathbb N})}(-1)$, respectively.  For each pair $(p,q)$
of non-negative integers, let $k^{p,q}$  denote the $k$-linear
span in $k^{({\mathbb N})}\oplus(k^{({\mathbb N})}(-1))$ of
$\{u_1,\dots,u_p\}\cup\{v_1,\dots,v_q\}$.

Note that one has $k^{p,q}\subseteq k^{p',q'}$ for $p\le p'$ and
$q\le q'$.
  \end{chunk}

  \begin{chunk}
    \label{ch:parameters}
Let $[1,p]$ denote the set $\{1,\dots,p\}$ of natural numbers with 
the natural order, and order the elements of\,\ $[1,e]\times[1,p]$ 
lexicographically:
 \begin{equation}
   \label{eq:order}
(l,n)< (l',n') \,\, \iff\,\, \text{$l<l'$ or ($l=l'$ and $n<n'$)}\,.
 \end{equation}

For $q\ge1$ we let $\mat q{ep}$ denote the set of\,\ $ep\times q$ matrices 
with entries in $k$, with rows indexed by the elements of\,\ $[1,e]\times[1,p]$ 
and columns by those of\,\ $[1,q]$.  \emph{We identify $\mat q{ep}$ with the
affine space $\mathbb{A}^{epq}_k$ over $k$, endowed with the Zariski topology};
by convention, we extend this identification to the case $q=0$.

For every subset $\bss\subseteq\seg{1,e}\times \seg{1,p}$ and $C\in\mat
q{ep}$, let $C_\bss$ denote the $|\bss|\times q$ submatrix of\,\ $C$ with
rows indexed by $\bss$; thus $C_{(l,n)}$ is the $(l,n)$th row of\,\ $C$.

When $|\bss|=q$, we form the following collection of matrices:
 \begin{equation}
   \label{eq:basic_open}
\mat q{ep}(\bss) = \{C\in\mat q{ep}\mid\det(C_\bss)\ne0\}\,.
 \end{equation}
This is a basic open subset of\,\ $\mat q{ep}$, which is empty when $q>ep$.
The subset
    \begin{equation}
  \label{eq:L0}
\LL^{0}_{p,q} = \{C\in \mat{q}{ep}\mid \rank C = q\}
   \end{equation}
is open $\mat q{ep}$, and is covered by the basic open subsets above:  
\begin{equation}
\label{eq:coverL0}
\LL^{0}_{p,q} = \bigcup_{\substack{\bss \subseteq\seg{1,e}\times \seg{1,p}\\ |\bss|=q}}\mat q{ep}(\bss)\,.
\end{equation}
  \end{chunk}

We parametrize short $R$-modules by means of their \emph{multiplication tables}.

  \begin{chunk}
    \label{ch:tables}
To each $R$-module $M$ with underlying vector space $k^{p,q}$ we 
associate the matrix $C^R=(c_{(l,n),h})$ in $\mat q{ep}$, with $(l,n)$th 
row defined by the equality
  \[
x_lu_n=\sum_{h=1}^q c_{(l,n),h}v_h
  \quad\text{for each}\quad
(l,n)\in[1,e]\times[1,p]\,.
  \]

Conversely, each matrix $C=(c_{(l,n),h})\in\mat q{ep}$ defines, through
the formula above, an action of\,\ $R$ on $k^{p,q}$ that turns it into an
$R$-module, called $R^{C}$.

The maps described above clearly are mutually inverse.  
     \end{chunk}
 
The correspondence in \ref{ch:tables} allows one to shuttle between
$R$-module structures on $k^{p,q}$ and $ep\times q$ matrices with elements
in $k$.  In particular, we identify $\LL^{0}_{p,q}$ with the set of\,\
$0$-step linear module structures on $k^{p,q}$.

Graded $R$-modules are often parametrized in terms of their \emph{minimal
presentations} over $R$.  This format is not used below, but we pause
to show that results on non-empty open loci faithfully translate between
parametrizations.

  \begin{chunk}
    \label{ch:grass}
For every matrix $B=(b_{(l,n),h'})$ in $\mat {(ep-q)}{ep}$, let
$\varkappa^B_1\col k^{ep-q}\to R_1\otimes_kk^p$ denote the homomorphism
of\,\ $k$-vector spaces given by the formula
  \[
w_{h'}\mapsto \sum_{(l,n)\in[1,e]\times[1,p]}b_{(l,n),h'}\,x_l\otimes
u_{n}\
  \quad\text{for}\quad h'=1,\dots,{ep-q}\,,
 \]
where $w_1,\dots,w_{ep-q}$ denote the standard basis of\,\ $k^{ep-q}$.

For every matrix $C=(c_{(l,n),h})$ in $\mat {q}{ep}$, let 
$\lambda^C_1\col R_1\otimes_kk^p\to k^q$ denote the 
homomorphism of\,\ $k$-vector spaces given by the formula
  \[
x_l\otimes u_{n}\mapsto
\sum_{h=1}^q c_{(l,n),h}v_h\
  \quad\text{for}\quad (l,n)\in[1,e]\times[1,p]\,.
  \]

Define an open subset of\,\ $\mat{(ep-q)}{ep}$ by setting
  \begin{align*}
\K_{p,q}^0&=\{B\in\mat{(ep-q)}{ep} \mid \rank_k(B)=ep-q\}\,.
  \end{align*}
  
The assignments $B\mapsto\operatorname{Im}(\varkappa(B))$ and
$C\mapsto\Ker(\lambda(C))$ define morphisms of algebraic varieties
to the Grassmannian of\,\ $(ep-q)$-dimensional subspaces of
$R_1\otimes_kk^p$:
 \[
\varkappa\col \K_{p,q}^0
\lra \Gr_{(ep-q)}(R_1\otimes_kk^p)
\longleftarrow \LL^0_{p,q} :\!\lambda
 \]
By construction, these maps above are morphisms of algebraic varieties. 

An important point here is that $\varkappa$ and $\lambda$ are \emph{open}.
This follows from a classical theorem of Chevalley, because
both maps are dominant (being surjective), the closed fibers of each one
have constant dimension (being isomorphic to some fixed affine space),
and Grassmann varieties are normal (being smooth).  Modern proofs of
Chevalley's Theorem are not easy to find.  Instead, we refer to a much
more general statement proved by Grothendieck in \cite[14.4.4(c)]{EGA43},
which contains the one used here; see \cite[6.15.1]{EGA42}.  We thank
Joseph Lipman for help with these references.
  \end{chunk}

Every matrix $B=(b_{(l,n),h'})$ in $\mat {(ep-q)}{ep}$ yields a 
homomorphism of graded $R$-modules 
$\varkappa^B\col R\otimes_kk^{ep-q}(-1)\to R\otimes_kk^p$, 
equal to $\varkappa^B_1$ in degree $1$.  The subsets 
  \begin{align*}
\K_{p,q}^1(R)&=\{B\in\mat{(ep-q)}{ep} \mid \text{$\varkappa^B_1$ is
injective and $\varkappa^B_2$ is surjective}\}\\
\LL_{p,q}^1(R)&=\{C\in\mat q{ep} \mid \text{ $\lambda^C_1$ is
injective} \}\,,
  \end{align*}
of\,\ $\mat{(ep-q)}{ep}$ are open.  One has $\hilb{\Coker(\varkappa^B)}s=p+qs$, 
so $\K_{p,q}^1(R)$ and $\LL_{p,q}^1(R)$ parametrize the same set of 
$R$-modules.  The parametrizations are interchangeable:

  \begin{lemma}
    \label{lem:grass}
When $k$ is an algebraically closed field and $R$ a standard graded
$k$-algebra the maps in \emph{\ref{ch:grass}} restrict to open morphisms 
of algebraic varieties
  \[
\varkappa\col \K_{p,q}^1(R)\lra\Gr_{(ep-q)}(R_1\otimes_kk^p)
\longleftarrow \LL_{p,q}^1(R) :\!\lambda
  \]
with the same image, so $U\subseteq\K_{p,q}^1(R)$ (respectively,
$U\subseteq\LL_{p,q}^1(R)$) is open or non-empty if and only
if\,\ $\lambda^{-1}\varkappa(U)\subseteq\LL_{p,q}^1(R)$
(respectively, $\varkappa^{-1}\lambda(U)\subseteq\K_{p,q}^1(R)$)~is$.$
   \end{lemma}

   \begin{proof}
Only the assertion concerning the images needs validation.  For each
matrix $C$ in $\mat{q}{ep}$, let 
$\lambda^C\col R(-1)\otimes_k\Ker(\lambda(C))\to R\otimes_kk^p$ be the
$R$-linear map, equal to $\lambda^C_1$ in degree $1$.  One has 
$C\in\LL_{p,q}^1(R)$ if and only if\,\ $\lambda^C_1$ is injective and 
$\lambda^C_2$ is surjective.  Comparison of definitions gives 
$\lambda(\LL_{p,q}^1(R))=\varkappa(\K_{p,q}^1(R))$.
  \end{proof}

\section{Linear loci}
  \label{sec:open}

Let $R$ be a standard graded $k$-algebra, and set
  \[
e=\rank_kR_1 \quad\text{and}\quad r=\rank_kR_2\,.
  \]

  \begin{chunk}
    \label{ch:full}
Fix positive integers $p$ and $q$. The \emph{linear locus} of\,\ $R$
in $\mat q{ep}$ is the subset
 \[
\LL_{p,q}(R)=\{C\in\mat q{ep} \mid \text{the $R$-module $R^C$ is Koszul}
\}\,,
 \]
where $R^C$ is the graded $R$-module defined in \ref{ch:tables}.
   \end{chunk}

Our goal is to identify conditions for $\LL_{p,q}(R)$ to have a
\emph{non-empty interior}; that is, for it to contain a non-empty
open subset.

\begin{theorem}
  \label{thm:functorial}
Let $R$ be a Koszul algebra and $p'\in[1,p]$ and $q'\in[q,ep]$ be integers.

If\,\ $\LL_{p',q}(R)$ or $\LL_{p,q'}(R)$ has a non-empty interior, then so 
does $\LL_{p,q}(R)$.
  \end{theorem}

In the proof we use the functoriality of the correspondence in \ref{ch:tables}:

 \begin{chunk}
  \label{ch:functorial}
For $p'\in[1,p]$, let $\iota\col k^{p',q}\hookrightarrow k^{p,q}$
denote the inclusion map.

As $k^{p',q}$ is a submodule for every $R$-module structure on 
$k^{p,q}$, we get a map 
  \begin{equation}
    \label{eq:submodules}
\iota^*\col\mat q{ep}\to\mat q{ep'}
  \end{equation}
of affine spaces over $k$, which is linear and surjective:  It sends each
$ep\times q$ matrix to the $(ep')\times q$ submatrix with
rows indexed by the pairs $(l,n)$ with $n\in[1,p']$.  

For $q'\in[q,ep]$, let $\pi\col k^{p,q'}\twoheadrightarrow k^{p,q}$
be the projection with $\pi(v_{h})=0$ for $h\ge q+1$.

Since $\Ker(\pi)$ is a submodule for every $R$-module structure on 
$k^{p,q'}$, we get a map 
  \begin{equation}
    \label{eq:quotientmodules}
\pi_*\col\mat {q'}{ep}\to\mat {q}{ep}
  \end{equation}
of affine spaces over $k$, which is linear and surjective: 
It sends every $ep\times q'$ matrix to its $ep\times q$ 
submatrix, whose columns are indexed by
the elements in $[1,q]$.
  \end{chunk}

\begin{lemma}
  \label{lem:functorial}
When $R$ is a Koszul algebra the maps \eqref{eq:submodules}
and \eqref{eq:quotientmodules} satisfy
  \begin{align*}
\pi_*(\LL_{p,q'}(R))\subseteq \LL_{p,q}(R)\supseteq(\iota^*)^{-1}(\LL_{p',q}(R))
  \end{align*}
 \end{lemma}

\begin{proof}
For $C\in\mat q{ep}$, the construction in \ref{ch:functorial} gives an exact 
sequence
$$
0\to R^{\iota^*(C)}\to R^C\to N\to 0
$$
where $N_j=0$ for $j\ne0$.  If\,\ $R^{\iota^*(C)}$ Koszul, then in the 
induced exact sequence
$$
\gext iRNkj\to\gext iR{R^C}kj\to\gext iR{R^{\iota^*(C)}}kj
$$
both extremal terms are zero for $j\ne i$, because $R$ is Koszul.

For $C'\in\mat {q'}{ep}$, the construction in \ref{ch:functorial} 
gives an exact sequence
$$
0\to L\to R^{C'}\to R^{\pi_*(C')}\to 0
$$
where $L_j=0$ for $j\ne1$.  When $R^{C'}$ is Koszul, in the induced exact  
sequence
$$
\gext{i-1}RLkj\to\gext iR{R^{\pi_*(C')}}kj\to\gext iR{R^{C'}}kj
$$
both extremal terms are zero for $j\ne i$, because $R$ is Koszul.
  \end{proof}

\begin{proof}[Proof of Theorem \emph{\ref{thm:functorial}}] Let
$U\subseteq\LL_{p',q}(R)$ be non-empty and open in $\mat{q}{ep'}$.
The subset $(\iota^*)^{-1}(U)$ of\,\ $\mat q{ep}$ is open, because $\iota^*$
is continuous, and not empty, because $\iota^*$ is surjective.  By Lemma
\ref{lem:functorial}, it is contained in $\LL_{p,q}(R)$.

Let $\sigma\col\mat q{ep}\to\mat{q'}{ep}$ be the
map that sends every $C\in\mat q{ep}$ to the $ep\times q'$ matrix,
obtained by the addition of zero columns with indices $q+1,\dots,q'$.
Let $U'\subseteq\LL_{p,q'}(R)$ be non-empty and open in $\mat{q'}{ep}$,
and pick $C'\in U'$.
In the affine subspace $W=C'+\sigma(\mat q{ep})$ of\,\ $\mat{q'}{ep}$, the set
$U'\cap W$ is non-empty and open.  Since $\pi_*|_W\col W\to\mat q{ep}$ is
a homeomorphism, $\pi_*(U'\cap W)$ is non-empty and open in $\mat q{ep}$.
By Lemma \ref{lem:functorial}, it is contained in $\LL_{p,q}(R)$.
  \end{proof}

Within a given parameter space $\mat q{ep}$, it is sometimes possible to transfer 
information between linear loci of different $k$-algebras.  We give an example.

\begin{proposition}
  \label{prop:homs}
Let $\rho\col R'\to R$ be a homomorphism of graded $k$-algebras.

If\,\ $\rho$ is a Golod homomorphism and $R'$ is Koszul, then for
all $p$ and $q$ one has
  \[
\LL_{p,q}(R)\supseteq\LL_{p,q}(R')\,.
  \]
\end{proposition}

 \begin{proof}
This follows from \cite[3.3]{AIS}, where it is proved that
if an $R$-module $M$ is Koszul when considered as a module 
over $R'$ via $\rho$, then $M$ is Koszul over $R$.
 \end{proof}

  \begin{chunk}
    \label{ch:partial}
We approximate the linear locus of\,\ $R$, \emph{from above}, by the sets
  \[
\LL_{p,q}^m(R)=\{C\in\mat q{ep} \mid R^C \text{ is $m$-step linear} \}\,,
 \]
defined for every integer $m\ge0$.  The following inclusions are evident: 
  \begin{equation}
    \label{eq:intersection}
\LL_{p,q}^m(R)\supseteq\LL_{p,q}^{m+1}(R)
  \quad\text{and}\quad
\LL_{p,q}(R)=\bigcap_{m\ges 0}\LL_{p,q}^m(R)\,.
  \end{equation}
   \end{chunk}

For completeness, we include the proof of a folklore result; stronger
ones have been communicated to us by David Eisenbud, Anthony
Iarrobino, and Clas L\"ofwall.

\begin{lemma}
  \label{lem:n-open}
When $R$ is Koszul $\LL^m_{p,q}(R)$ is open in $\mat q{ep}$ for every $m\ge0$.
 \end{lemma}

\begin{proof}
Pick a matrix $C$ in $\mat q{ep}$.

One has $\gb ijR{R^C}=\rank_k\gtor iR{R^C}kj$, so we fix $m\ge0$ and
prove that the subset of\,\ $\mat q{ep}$, defined by $\gtor iR{R^C}kj=0$
for $j\ne i\le m$, is open.

Let $G$ be a minimal free resolution of\,\ $k$ over $R$.  As $R$
is Koszul, we may assume $G_i=R(-i)^{b_i}$ with $b_i=\beta^R_i(k)$.
As ${R^C}$ is short, one has $({R^C}\otimes_RG_i)_j=0$ for $j\ne i,i+1$.
This yields $\gtor iR{R^C}kj=0$ for $j\ne i,i+1$, and an exact sequence
  \[
0\to\gtor iR{R^C}ki\to
(R^C_0)^{b_i}\xra{\,\delta_i\,}(R^C_1)^{b_{i-1}}\to\gtor{i-1}R{R^C}ki\to0
   \]
of\,\ $k$-vector spaces for every $i\ge0$, where
$\delta_i=(R^C_0\otimes_R\dd_i^G)_i$.

For each $i$, let $G^{(i)}$ denote the standard basis of\,\ $R(-i)^{b_i}$
over $R$.  In these bases $\dd^G_i$ is given by a matrix of linear forms
in $x_1,\dots,x_e$.  In the $k$-bases
  \begin{align*}
\{u_n\otimes g^{(i)}&\mid n=1,\dots,p\,;\,g^{(i)}\in G^{(i)}\}
  \quad\text{and}
    \\
\{v_h\otimes g^{(i-1)}&\mid h=1,\dots,q\,;\,g^{(i-1)}\in G^{(i-1)}\}
    \end{align*}
of\,\ $(R^C_0)^{b_i}$ of\,\ $(R^C_1)^{b_{i-1}}$, respectively,
the map $\delta_i$ then is given by a matrix of linear forms in the
elements $c_{(l,n),j}$ from the multiplication table in \ref{ch:tables}.
The condition $\gtor{i}R{R^C}k{i+1}=0$ for $i\le m$ is equivalent to the
surjectivity of\,\ $\delta_i$ for $i\le m+1$.  The latter condition means
that some maximal minor of the matrices $\delta_1,\dots,\delta_{m+1}$
is different from $0$, and so determines an open subset of $\mat q{ep}$.
  \end{proof}

  \section{Periodic linear modules}
    \label{sec:per}

In this section $R$ denotes a standard graded $k$-algebra.  We say that
a graded $R$-module $M$ is \emph{linear of period $2$} if there is an
exact sequence of graded $R$-modules
  \begin{equation}
    \label{eq:per}
0\to M(-2)\to R^p(-d-1)\to R^p(-d)\to M\to0
  \end{equation}
Splicing suitable shifts of the exact sequence above, one sees that $M$
is linear.  

We explore the interplay of periodicity, linearity, and shortness.  When 
$N$ is an $R$-module, $\Omega^R_i(N)$ denotes the $i$th module of 
syzygies in a minimal free resolution of $N$.  Assuming $R_3=0$, Lescot 
\cite[3.4]{Le} established part (1) of the next theorem.

\begin{theorem}
  \label{thm:per}
Let $R$ be standard graded $k$-algebra with $\rank_kR_1=e\ge1$, let $M$
be a non-zero $R$-module with $\inf M=d$, and $p$ a positive integer.
 \begin{enumerate}[\rm(1)]
 \item
The following conditions are equivalent.
 \begin{enumerate}[\quad\rm(i)]
  \item
$\po M{s,t}=ps^d\cdot(1-st)^{-1}$ and $M$ is short.
  \item
$M$ is linear over $R$ with $\hilb{M}{s}=(1+(e-1)s)\cdot ps^{d}$, and one has
\newline $\hilb{R}{s}=(1+(e-1)s)\cdot(1+s)$.
   \end{enumerate}
They imply that $R$ is Koszul.
   \medskip
 \item
If $Q$ is a standard graded $k$-algebra with $\rank_kQ_1=e$ and
$\psi\col Q\to R$ is a surjective homomorphism of algebras, then the following
conditions are equivalent.
 \begin{enumerate}[\quad\rm(i)]
  \item[\rm(iii)]
$\po[Q]M{s,t}=ps^{d}\cdot(1-(st)^{c+1})\cdot(1-st)^{-1}$
for some $c\ge1$, and $M$ is short.
  \item[\rm(iv)]
$M$ is linear over $Q$ with $\hilb{M}{s}=(1+(e-1)s)\cdot ps^{d}$, and one has
\newline $\hilb{Q}{s}=(1+(e-1)s)\cdot(1-s)^{-1}$.
  \item[\rm(v)]
$\po[Q]M{s,t}=ps^{d}\cdot(1+st)$ and
$\hilb{Q}{s}=(1+(e-1)s)\cdot(1-s)^{-1}$.
 \end{enumerate}
   \medskip
They imply that $Q$ is Koszul, Golod, and Cohen-Macaulay of dimension $1$.
   \medskip
 \item
If $g\in Q_2$ is a non-zero-divisor and $\Ker(\psi)=(g)$ in \emph{(2)}, then
the condition
   \medskip
 \begin{enumerate}[\quad\rm(i)]
  \item[\rm(vi)]
$M$ is linear of period $2$ over $R$ and $\hilb{R}{s}=(1+(e-1)s)\cdot(1+s)$\,.
 \end{enumerate}
   \medskip
satisfies the implications \emph{(v)}$\implies$\emph{(vi)}$\iff$\emph{(i)},
and also \emph{(v)}$\iff$\emph{(vi)} if $e\ne2$.
 \end{enumerate}
  \end{theorem}

\begin{proof}
When $M$ is short we write $\hilb Ms$ in the form $s^d(a+bs)$.  

(1) (i)$\implies$(ii).
The algebra $R$ is Koszul by Theorem \ref{thm:test}, so using
\eqref{eq:linearg} we get
  \[
p\hilb R{st}= \frac{p\hilb M{st}}{t^d\cdot\po M{s,-t}} = (a+bst)(1+st)\,.
  \]
The expressions for $\hilb Ms$ and $\hilb Rs$ follow by comparing degrees
and coefficients.

(ii)$\implies$(i).  
This implication follows directly from \eqref{eq:linearg}.

\medskip

(2) (iii)$\implies$(iv).
The algebra $Q$ is Koszul by Theorem \ref{thm:test}, so \eqref{eq:linearg}
gives
  \[
p\hilb Q{st}= \frac{p\hilb M{st}}{t^d\cdot\po[Q]M{s,-t}} =
\frac{(a+bst)(1+st)}{1-(-st)^{c+1}}\,.
  \]
Recall that $\hilb Q{s}$ can be written as $h(s)/(1-s)^n$ with $h(s)$
in $\BZ[s]$ satisfying $h(1)\ne0$, and $n=\dim Q$.  Setting $t=1$
in the formula above, we obtain
  \[
ph(s)(1-(-s)^{c+1})=(a+bs)(1+s)(1-s)^n\in\BZ[s]\,.
  \]
Comparing orders of vanishing at $s=1$, we get $c=1=n$, hence
$ph(s)=a+bs$.  The desired expressions for $\hilb Ms$ and $\hilb Qs$
follow, along with $\po[Q]M{s,t}=(1+st)\cdot ps^d$.  Thus, $M$ has
projective dimension~$1$, which entails $\depth Q\ge1$.  The expression
for $\hilb Ms$ yields $\dim Q=1$, so $Q$ is Cohen-Macaulay; it also
shows that $Q$ has multiplicity $e$; as $\edim Q=e$ and $\dim Q=1$,
this is the minimal possible value.

(iv)$\implies$(v)$\implies$(iii).  
These implications follows directly from \eqref{eq:linearg}.

\medskip

(3) The isomorphism $R\cong Q/(g)$ with $g$ a non-zero-divisor
in $Q_2$ implies
   \begin{equation}
     \label{eq:square}
 \hilb{R}s=(1-s^2)\hilb{Q}s\,.
  \end{equation}

(v)$\implies$(vi).
By hypothesis, there is an exact sequence of $Q$-modules
  \[
0\to Q^p(-1)\to Q^p\to M\to 0
  \]
The resolution $0\to Q(-2)\xra{g}Q\to0$ of $R$ over $Q$ yields $\tor
[1]QR{M}\cong M(-2)$, so application of $\tor QR{-}$ to the exact
sequence above yields an exact sequence of the form \eqref{eq:per}.
The expression for $\hilb Rs$ comes from formula \eqref{eq:square}.

(vi)$\implies$(i).
{From} the exact sequence \eqref{eq:per} we obtain
  \[
\hilb{M}{s}=\hilb{R}{s}\cdot ps^d(1+s)^{-1}=(1+(e-1)s)\cdot ps^d\,.
  \]
It follows that $M$ admits no direct summand isomorphic to $R$, so the
linear free resolution of $M$ over $R$, obtained by splicing suitable
shifts of \eqref{eq:square}, is minimal.

(i)$\implies$(v) when $e\ne2$.  
If $e=1$, then one has $Q\cong k[x]$ and $M\cong k^p$.

Assume $e\ge3$.  From \cite[3.3.4]{Av:barca}, one gets $\beta^Q_i(M)\le2p$
for each $i$. The ring $Q$ is Cohen-Macaulay of dimension $1$.  One  
gets $\hilb Qs=(1+(e-1)s)(1-s)^{-1}$ from \eqref{eq:square}, so $Q$
has embedding dimension $e$ and multiplicity $e$.  Thus,
$Q$ has \emph{minimal multiplicity}, and so is Golod; see
\cite[5.2.8]{Av:barca}.  This implies that $M$ has finite projective
dimension over $Q$; see \cite[5.5.3(5)]{Av:barca}.  As $\depth Q=1$,
there is an exact sequence
  \[
0\to \bigoplus_{j=1}^nQ^{r_j}(-j)\to Q^p(-d)\to M\to 0
  \]
The expressions for $\hilb Qs$, noted above, and for $\hilb Ms$, from 
(ii), then yield
  \[
ps^d(1+(e-1)s)=\hilb Ms
=\bigg(ps^d-\sum_{j=1}^nr_js^{d+j}\bigg)\frac{1+(e-1)s}{1-s}\,.
  \]
We get $ps^d(1-s)=ps^d-\sum_{j=1}^nr_js^j$, whence $r_j=0$
for $j\ne d+1$ and $r_{d+1}=p$.

(i)$\implies$(vi) when $e=2$.  
We may assume $d=0$.

By (1) the algebra $R$ is Koszul, hence quadratic, and thus $R\cong
k[x,y]/(f,g)$ with $f,g$ a regular sequence of quadrics.

The hypothesis on $\po M{s,t}$ give an exact sequence
  \[
0\to N\to R^p(-3)\xra{\alpha}R^p(-2)\xra{\beta}R^p(-1)
\xra{\gamma}R^p\to M\to0
  \]
As $R$ is complete intersection, one has $N\cong\Coker(\alpha)(-2)$ by \cite[4.1]{Ei}.

Set $(-)^*=\Hom_R(-,R)$.  Since $R$ is self-injective, we get an exact sequence 
  \[
0\to M^*\to R^p\xra{\gamma^*}R^p(1)\xra{\beta^*}R^p(2)\xra{\alpha^*}R^p(3)
\to N^*\to0
  \]
and an isomorphism $\Ker(\alpha^*)\cong N^*(-2)$.  From these data, we obtain 
  \[
M^*\cong\Omega^2_R(\Ker(\alpha^*))\cong \Omega^2_R(N^*(-2))\cong N^*(-4)
  \]
As all $R$-modules are reflexive,
we get $M\cong N(4)\cong\Coker(\alpha)(2)
\cong\Ker(\gamma)(2)\,.$
  \end{proof}

We use modules of period $2$ to study the linear locus $\LL_{p,(e-1)p}(R)$.

 \begin{chunk}
    \label{ch:per}
We approximate the linear locus of $R$, \emph{from below}, by the set
  \[
\LP^2_{p,q}(R)=\{C\in\LL_{p,q}(R)\mid R^C\text{ is of period }2\}\,.
  \]
  \end{chunk}

\begin{corollary}
  \label{cor:per}  
Assume $\hilb{R}{s}=1+es+(e-1)s^2$ with $e\ge1$, and $R\cong Q/(g)$ for
some standard graded algebra $Q$ and non-zero-divisor $g\in Q_{2}$.

For each positive integer $p$ the set $\LL_{p,(e-1)p}^2(Q)$ is open
in $\mat{(e-1)p}{ep}$ and there the following inclusions hold,
with equality when $e\ne2$:
  \begin{equation*}
   %\label{eq:per2}
\LL_{p,(e-1)p}(R)=\LP^2_{p,(e-1)p}(R)
\supseteq\LL_{p,(e-1)p}(Q)=\LL_{p,(e-1)p}^2(Q)\,.
  \end{equation*}

If $Q_1$ contains a non-zero-divisor (e.g., if\,\ $k$ is infinite),
then $\LL_{p,(e-1)p}(Q)\ne\varnothing$.
   \end{corollary}

  \begin{proof}
Every short $Q$-module is annihilated by $g$, for degree reasons,
and so is also an $R$-module.  With this remark, Theorem \ref{thm:per}
implies the inclusion and the equalities.  The set $\LL_{p,(e-1)p}^2(Q)$
is open by Lemma~\ref{lem:n-open}.

If\,\ $h\in Q_{1}$ is a non-zero divisor, then for $N=(Q/(h))^{p}$
one has an exact sequence
  \[
0\to Q(-1)^p\xra{\, h\,} Q^p\to N\to0
  \]
that gives $\hilb{N}s=(p-ps)\hilb Qs=p+(e-1)ps$; thus, $N$ is
in $\LL^2_{p,(e-1)p}(Q)$.
  \end{proof}

When $e=2$, neither the implication (v)$\implies$(vi) in Theorem \ref{thm:per},
nor the inclusion in Corollary \ref{cor:per} can be reversed in general:

\begin{example}
  \label{ex:per}  
Set $Q=k[x,y]/(x^2)$ and $R=Q/(y^2)$.  The $R$-module $M=R/(x)$ then
is linear of period $2$, as demonstrated by the exact sequence
  \begin{gather*}
0\to M(-2)\to R(-1)\xra{x}R\to M\to0
  \end{gather*}
However, as a $Q$-module $M$ has infinite projective dimension and is not linear.
   \end{example}

For cyclic modules, we have an additional criterion for openness.

\begin{proposition}
  \label{prop:per}  
If $R$ is a Koszul algebra with $\hilb{R}{s}=1+es+(e-1)s^2$, then 
$\LL_{1,e-1}(R)$ is an open subset of $\mat{(e-1)}{e}$, and there are
equalities
  \[
\LL_{1,e-1}(R)=\LP^2_{1,e-1}(R)=\LL^2_{1,e-1}(R)\,.
  \]
   \end{proposition}

  \begin{proof}
In view of \eqref{eq:intersection}, it suffices to show that if
$\LL^2_{1,e-1}(R)$ is not empty, then it is contained in $\LP^2_{1,e-1}(R)$.
Pick $C$ in $\LL_{1,e-1}^2(R)$ and set $M=R^C$.  One then has $\hilb
Ms=1+(e-1)s$, which implies $\po{M}{s,t}\equiv 1+st+(st)^2 \pmod{t^{3}}$,
by Lemma~\ref{lem:m-step}.  Thus, for appropriate $a,b\in R_1$ there is
an exact sequence
  \[
R(-2)\xrightarrow{b} R(-1)\xrightarrow{a} R\to M\to 0
  \]
It gives $bR=(0:a)$ and hence $aR\subseteq (0:(0:a))=(0:b)$.
The resulting relations
  \[
\rank_k(aR)\le \rank_k(0:b)=\rank_k(R/bR)=\rank_k(aR)
  \]
imply $(0:b)=aR$, hence $bR\cong R/aR$.  As a result, we obtain an exact
sequence
  \[
0\to M(-2)\to R(-1)\xrightarrow{a} R\to M\to0
 \qedhere
  \] \end{proof}

\section{Algebras with Conca generators}
  \label{sec:conca}

In this section $R$ is a standard graded algebra, and we set
  \[
e=\rank_kR_1 \quad\text{and}\quad r=\rank_kR_2\,.
  \]
A \emph{Conca generator} of\,\ $R$ is a non-zero element $x\in R_1$
with $xR_1=R_2$ and $x^2=0$, cf.\ \cite{AIS}.  We collect relevant facts
about algebras containing such an element.

    \begin{chunk}
    \label{ch:conca1}
When $k$ is algebraically closed, generic quadratic algebras with $r\le
e-1$ have a Conca generator; see the proof of \cite[Thm.\,10]{Co}.
 \end{chunk}

    \begin{chunk}
    \label{ch:conca2}
If\,\ $x$ is a Conca generator for $R$, then one clearly has $r\le e-1$
and $R_3=0$.  

Furthermore, the algebra $R$ then is Koszul by \cite[2.7]{CRV},
see also \cite[Lem.\,2]{Co} or \cite[1.1]{AIS}, and every $R$-module $M$
with $xM=0$ is linear by \cite[4.2]{AIS}.
 \end{chunk}

  \begin{chunk}
    \label{ch:presentation}
The algebra $R$ has a Conca generator if and only if it has a presentation
$R=k[x_1,\dots,x_{e}]/I$ with defining ideal $I$ generated by
   \begin{align}
 \label{eq:monomials}
x_{l}x_e
   &\quad\text{for}\quad r+1\le l\le e\,,
  \\
 \label{eq:quadrics}
x_lx_{l'}-\sum _{h=1}^r a_{l,l';h}x_hx_e
   &\quad\text{for}\quad 1\le l\le l'\le e-1\,.
  \end{align}
The class of\,\ $x_e$ is thus a Conca generator for $R$.
  \end{chunk}

Theorem \ref{i2} from the introduction is contained in Propositions 
\ref{prop:(e-1)p}, \ref{prop:e-1}, and \ref{prop:(e-r)p}.  In their proofs, 
we use the order on $\seg{1,e}\times \seg{1,p}$ defined in \eqref{eq:order}.

\begin{proposition}
\label{prop:(e-1)p}
Let $R$ be a standard graded algebra with a Conca generator.  

For all positive integers $p,q$ with $q\le (e-1)p$ one has 
$\LL_{p,q}(R)\ne\varnothing$.

If $r=e-1$, then $\LL_{p,q}(R)\ne\varnothing$ implies $q\le (e-1)p$,
one has $\LP^2_{p,(e-1)p}(R)\ne\varnothing$ for each $p\ge1$, the 
set $\LP^2_{1,e-1}(R)$ is open in $\mat {e-1}1$, and 
$\LP^2_{1,e-1}(R)=\LL_{1,e-1}(R)$.
  \end{proposition}

  \begin{proof}
Let $\bss$ be the set consisting of the $q$ smallest elements of
$[1,e]\times[1,p]$, and $C\in\mat q{ep}$ the matrix with $C_{\bss}$ equal
to the $q\times q$ unit matrix and $C_{(l,n)}=0$ for $(l,n)\notin\bss$.
The condition $q\leq (e-1)p$ implies $(e,n)\notin\bss$ for $n=1,\dots,p$,
hence $x_eR^C=0$.  Now recall that every $R$-module $M$ with $x_eM=0$
is linear, see \ref{ch:conca2}.

Assume $r=e-1$.  When $\LL_{p,q}(R)\ne\varnothing$, Corollary
\ref{n-injective} yields $q\le (e-1)p$.  If $x$ is a Conca generator, 
a rank count gives $(0:x)=xR$, so $(R/xR)^p$ is in $\LP^2_{p,(e-1)p}(R)$.
The sets $\LL_{1,e-1}(R)$ and $\LP^2_{1,e-1}(R)$
are equal and open by Proposition \ref{prop:per}.
  \end{proof}

\begin{proposition}
\label{prop:e-1}
Let $R$ be a standard graded algebra with a Conca generator.

If\,\ $q\le e-1$, then the interior of\,\ $\LL_{p,q}(R)$ is not empty.
 \end{proposition}

 \begin{proof}
Assume first $r=e-1$.  Now $R$ is Koszul, see \ref{ch:conca2}, so
$\LL_{1,e-1}(R)$ is open and non-empty by Propositions \ref{prop:e-1}
and \ref{prop:(e-1)p}.  Using Theorem \ref{thm:functorial} we extend
this conclusion first to $\LL_{p,e-1}(R)$ for arbitrary $p$, then to
all $\LL_{p,q}(R)$ with $q\le e-1$.

Assume next $r<e-1$.  Set $R'=k[x_1,\dots, x_e]/I'$ where $I'$ is
generated by $x_e^2$ and the polynomials in \eqref{eq:quadrics}. One has
$\hilb{R'}s=(1+(e-1)s)\cdot(1+s)$ and the class of $x_e$ in $R'$
is a Conca generator, so $R'$ is Koszul; see \ref{ch:conca2}.  The  map $R'\to R$
is Golod; see the proof of \cite[3.2]{AIS}.  In $\mat q{ep}$ this gives
$\LL_{p,q}(R)\supseteq\LL_{p,q}(R')$, by Proposition \ref{prop:homs},
and $\LL_{p,q}(R')$ has a non-empty interior by the case already settled.
 \end{proof}

\begin{proposition}
\label{prop:(e-r)p}
Let $R$ be a standard graded algebra with a Conca generator.

If\,\ $p,q\in\mathbb{N}$ satisfy $q\leq (e-r)p$, then $\LL_{p,q}(R)$
contains the non-empty open set $\LL^1_{p,q}(R)\cap\mat q{ep}(\bss)$,
where $\bss$ consists of the $q$ largest elements of\,\
$[1,e]\times[1,p]$.
 \end{proposition}

\begin{proof}
Both $\mat q{ep}(\bss)$ and $\LL^1_{p,q}(R)$ are open and nonempty, see
\ref{ch:parameters} and Lemma \ref{lem:n-open}, so their intersection
in the affine space $\mat q{ep}$ has the same properties.  It remains
to prove that each $R$-module $M=R^C$ in this intersection is Koszul.

By definition, the rows of the matrix $C\in\mat q{ep}(\bss)$ 
form a basis of the row space of\,\ $C$, so they determine
elements $d_{(l,n),(l',n)}\in k$, such that
  \begin{equation*}
 \label{relations}
C_{(l,n)}=\sum_{(l',n')\in\bss} d_{(l,n),(l',n')}C_{(l',n')}
  \quad\text{for all}\quad (l,n)\in [1,e]\times[1,p]\,.
  \end{equation*}
In degrees $0$ and $1$ the following sequence of graded $R$-modules
  \begin{align*}
&\hphantom{R\otimes_k\bigoplus_{n\in[1,p]}ku}
1\otimes u_n \longmapsto u_n
  \\
R\otimes_k\bigoplus_{(l,n)\notin\bss}(kx_l)\otimes_k(ku_n)
&\lra R\otimes_k\bigoplus_{n\in[1,p]}ku_n
\lra M\lra 0
  \\
1\otimes x_l\otimes u_n &\longmapsto
x_l\otimes u_n-\sum_{(l',n')\in \bss}d_{(l,n),(l',n')}(x_{l'}\otimes u_{n'})
  \end{align*}
is exact by construction, so it is exact because $M$ has a linear 
presentation.

The presentation of\,\ $M$ described above yields one for $S=R\ltimes M$,
in the form
  \[
S\cong k[x_1,\dots x_e,y_1,\dots y_p]/J\,,
  \]
where $J$ is generated by the polynomials in \eqref{eq:monomials},
\eqref{eq:quadrics}, and by those below:
  \begin{alignat}{3}
    \label{eq:new}
&y_ny_{n'}
  &&\quad\text{for}\quad &&1\le n,{n'}\le p\,,
    \\ \label{eq:new2}
&x_ly_{n}-\sum_{(l',n')\in \bss}d_{(l,n),(l',n')}x_ly_n
  &&\quad\text{for}\quad &&(l,n)\notin \bss\,.
 \end{alignat}
For monomials in $x_1,\dots,y_p$ we use the reverse degree-lexicographic
order with
  \[
y_1>\dots >y_p>x_1>\dots >x_e\,.
  \]
Thus, $x_ly_n> x_{l'}y_{n'}$ is equivalent to $(l,n)<(l',n')$, so the
choice of\,\ $\bss$ implies:
  \[
x_ly_n>x_{l'}y_{n'}\quad\text{holds when}\quad (l,n)\notin \bss
\quad\text{and}\quad (l',n')\in \bss\,.
  \]
For the chosen generators of\,\ $J$, this gives the following list of
initial terms:
  \begin{xalignat*}{10}
\qquad &&&x_lx_{l'}
  &&\text{for} &&1\le l\le {l'}\le e-1\,,
\qquad&&x_ly_n
  &&\text{for} &&(l,n)\notin \bss\,,&&& \\
\qquad &&&x_{l}x_e
  &&\text{for} &&r+1\le l\le e\,,
&&y_ny_{n'}
  &&\text{for} &&1\le n,{n'}\le p\,.&&&
   \end{xalignat*}

Set $T=k[x_1,\dots,x_e, y_1,\dots, y_p]/L$, where $L$ denotes the ideal
generated by the monomials listed above.  We claim that $L$ contains
all monomials of degree~$3$.  Only those of the form $x_lx_{l'}y_n$ with
$l\le l'$ need attention.  Unless $l\le r$ and $l'=e$ hold, $x_lx_{l'}$
is in $L$, hence so is $x_lx_{l'}y_n$. For $l\le r$ the hypothesis $q\le
(e-r)p$ and the choice of\,\ $\bss$ imply $(l,e)\notin\bss$, so $x_ly_n$
is $L$, hence so is $x_lx_ey_n$.

We just proved that $T_3$ is zero.  Counting non-zero monomials in $T$, we
get $\hilb Ts=1+(e+p)s+(e-1+q)s^2$.  This gives $\hilb Ts=\hilb Ss$, which
implies that the polynomials in \eqref{eq:monomials}, \eqref{eq:quadrics},
and (\ref{eq:new}), (\ref{eq:new2}) form a Gr\"obner basis for $S$.
It follows that $S$ is Koszul, see \cite[\S4]{Fr0}, hence so is $M$,
by Proposition \ref{koszul-rings-modules}.
 \end{proof}

\section{Open sets of linear modules}
  \label{sec:loci}

As before, $R$ denotes a standard graded $k$-algebra, and we set
  \[
e=\rank_kR_1 \quad\text{and}\quad r=\rank_kR_2\,.
  \]
We assume $e\ge1$, and let $p\ge1$ and $q\ge0$ denote integers.  

Here our goal is to record various instances when the linear locus
$\LL_{p,q}(R)$ is open and non-empty.  The minimal admissible values of\,\
$r$ and $e$ are easily disposed of:

\begin{chunk}
\label{ch:e1}
If $e=1$ and $r=1$, then $\LL_{p,q}(R)=\varnothing$ for all $p$ and $q$.

Indeed, these conditions imply $R\cong k[x]$ or $R\cong k[x]/(x^n)$
for some $n\ge3$. 
  \end{chunk}

\begin{chunk}
\label{ch:r0}
When $e\ge1$ and $r=0$ $R$ is Koszul, $\LL_{p,q}(R)=\LL^0_{p,q}(R)$,
and this set is open in $\mat q{ep}$ for all $p$ and $q$; one has
$\LL_{p,q}(R)\ne\varnothing$ if and only if\,\ $q\leq ep$.
  \end{chunk}

In view of the preceding remarks, we henceforth focus on the case $e\ge2$.

The proofs of the next two propositions draw on most results in the paper.

  \begin{proposition}
    \label{prop:gor}
If\,\ $e\ge2$ and $R$ is short and Gorenstein, then it is Koszul.

When $e=2$, for each pair $(p,q)$ one has $\LL_{p,q}(R)=\LL^{q-1}_{p,q}(R)$,
this set is open in $\mat q{ep}$, and $\LL_{p,q}(R)\ne\varnothing$ if 
and only if\,\ $q\le p$.

When $e\ge3$, for each pair $(p,q)$ one has $\LL_{p,q}(R)=\LL^{m}_{p,q}(R)$
for some $m$, and this set is open in $\mat q{ep}$; if there exists an non-zero element
$x\in R_1$ with $x^2=0$ (in particular, if $k$ is algebraically closed), then
$\LL_{p,q}(R)\ne\varnothing$ for $q\leq (e-1)p$.
  \end{proposition}

  \begin{proof}
For a proof that $R$ is Koszul see, for instance, \cite[2.7]{CRV}, or 
\cite[4.1]{AIS}. 

The set $\LL_{p,q}^m(R)$ is open in $\mat q{ep}$ for all $m$, $p$,
and $q$, see Lemma~\ref{lem:n-open}.  In particular, the openness of
$\LL_{p,q}(R)$ follows from the other assertions.

For each $i\in\BZ$, set $b_{i} = \beta_{i}^{R}(k)$ and
$M_{(i)}=\Hom_R(\Omega^{i}_{R}(k),R)(1-i)$.  One has:
   \begin{enumerate}[\quad\rm(1)]
 \item
$b_{i}> b_{i-1}$ for every $i\ge0$; moreover, $b_{i}=i+1$ when $e=2$.
\item $\hilb{M_{(i)}}s=b_{i-1} + b_{i}s$.
 \item
If\,\ $N$ is an indecomposable non-Koszul module, then $N\cong M_{(i)}$
for some~$i\ge1$.
  \end{enumerate}
Indeed, the recurrence relation $b_{i+1}=eb_{i}-b_{i-1}$, for $i\geq 2$,
given by \eqref{eq:koszul}, implies (1).  Parts (2) and (3) are contained
in \cite[4.6]{AIS}(2).  Next we prove:
  \begin{enumerate}[\quad\rm(1)]
 \item[\rm(4)] $M_{(i)}$
is $(i-1)$-step linear, but not $i$-step linear.
  \end{enumerate}

Indeed, since $R$ is Koszul one obtains an exact sequence
 \[
0\to\Omega^{i}_{R}(k)\to R(-i+1)^{b_{i-1}}\to\cdots\to R(-1)^{b_{1}}
\to R \xra{\varepsilon} k\to0\,,
 \]
from a minimal free resolution of\,\ $k$ over $R$.   Now $\Hom_{R}(-,R)$
is exact because $R$ is Gorenstein, and $\Hom_R(k,R)\cong k(-2)$ as
$R_{2}\cong k$, so we get an exact sequence
 \[
0\to k(-2)\xra{\eta} R\to R(1)^{b_{1}}\to\cdots\to R(i-1)^{b_{i-1}}
\to \Hom_R(\Omega^{i}_{R}(k),R)\to0
 \]
It gives for $M_{(i)}$ a minimal free resolution depicted below, which
proves (4):
 \[
\cdots\to R(-i-1)\to R(-i+1)\to R(-i+2)^{b_{1}}\to\cdots\to R^{b_{i-1}}
\to 0
 \]

Choose now, by (1), an integer $m$ so that $b_{i}> q$ holds for $i>m$;
by the same token, pick $m=q-1$ when $e=2$.  If\,\ $\LL^{m}_{p,q}(R)$
contains a module $M$ that is not Koszul, then some indecomposable direct
summand $N$ of\,\ $M$ is not Koszul. By (3), we have $N\cong M_{(i)}$
for some $i\geq 1$, so $M_{(i)}$ is $m$-step linear.  Now (4) implies
$i> c$, hence $b_{i}>q$ by the choice of\,\ $m$. On the other hand, for
the submodule $M_{(i)}$ of\,\ $M$ we get  $b_{i} \leq q$ from (2). The
contradiction implies $\LL_{p,q}(R)=\LL^{m}_{p,q}(R)$, as desired.

When $e=2$, Corollary~\ref{n-injective} gives $\LL_{p,q}(R)=\varnothing$
when $q>p$.  Thus, we assume $q\leq p$ and set out to prove 
$\LL_{p,q}(R)\ne\varnothing$.  In view of Lemma~\ref{lem:functorial}, 
we may restrict to the case $q=p$.  Since $R$ is
Koszul, we have $R\cong k[x,y]/(f,g)$ with $f,g$ a regular 
sequence in $Q_2$.  Set $Q=k[x,y]/(f)$.  Corollary~\ref{cor:per}
shows that it suffices to exhibit a non-zero-divisor $h\in Q_1$.  If $f$ 
is irreducible, then $Q$ is an integral domain; take $h=x$.  Else, 
$f$ is a product of two linear forms, so after a change of variables
we may assume $(f)=(x^2)$ or $(f)=(xy)$; in either case, pick $h=x+y$.

Finally, as $R$ is Gorenstein every non-zero element $x\in R_1$ with 
$x^2=0$ evidently is a Conca generator; such an $x$ exists when $k$ 
is algebraically closed, see for instance,  \cite[Lem.\,3]{Co}.  Now
Proposition~\ref{prop:(e-1)p} gives $\LL_{p,q}(R)\ne \varnothing$ 
for $q\leq (e-1)p$.
 \end{proof}

  \begin{proposition}
    \label{prop:per3}
Assume that $R$ is quadratic, with $\hilb Rs=1+es+(e-1)s^2$. 

If $e=2$, or if $e=3$ and $k$ is infinite, then the following hold.
  \begin{enumerate}[\rm(1)]
    \item
There is an isomorphism $R\cong Q/(g)$ for a Koszul $k$-algebra $Q$
and a non-zero-divisor $g\in Q_2$; in particular, $R$ is Koszul.
    \item
For every positive integer $p$ one has $\LL_{p,(e-1)p}(R)=\LP_{p,(e-1)p}^2(R)$,
this set is open in $\mat{(e-1)p}{ep}$, and is not empty.
  \end{enumerate}
 \end{proposition}

 \begin{proof}
(1) When $e=2$ one has $R\cong k[x,y]/(f,g)$ with $f,g$ a regular sequence
of quadrics; take $Q=k[x,y]/(f)$.

When $e=3$, write $R\cong k[x,y,z]/I$ with $I$ an ideal minimally
generated by $4$ quadrics.  Assuming $hR_1=0$ for some $h\in R_1$ with
$h\ne0$, we get a quadratic algebra $S=R/(h)$ with $\hilb Ss=1+2s+2s^2$;
this is impossible, and so we get $(0:\fm)=R_2$.  Thus, $R$ is an
\emph{almost complete intersection of codimension $3$ and type~$2$}.
In the local case, such rings are described by a structure theorem of
Buchsbaum and Eisenbud, see \cite[5.4]{BE}.  This is a corollary of
\cite[5.3]{BE}, whose proof refers to a general position argument to
find generators $f_1,f_2,f_3,f_4$ of $I$, every $3$ of which form a
regular sequence.  The hypothesis that $k$ is infinite allows one to
find the $f_i$ as $k$-linear combinations of the original quadrics.
The rest of the proof of \cite[5.3]{BE} and that of \cite[5.4]{BE}
now yield $\{i_1,i_2,i_3\} \subset[1,4]$ and $f\in Q_2$, such that
$I=(f_{i_1},f_{i_2},f_{i_3},f)$ and $f_{i_1}$ is a non-zero-divisor on
the algebra $Q=k[x,y,z]/(f_{i_2},f_{i_3},f)$.

(2)  In view of (1), Corollary \ref{cor:per} applies.  It yields
$\LL_{p,(e-1)p}(R)=\LP_{p,(e-1)p}^2(R)$, shows that this set is open in
$\mat{(e-1)p}{ep}$, and also that it is non-empty when $e=3$.  When $e=2$,
we get $\LL_{p,p}(R)\ne\varnothing$ from Proposition \ref{prop:gor}.
  \end{proof}

To finish, we show that when $R$ needs more that $3$ generators,
there exists no similar description of the locus of modules 
with constant Betti numbers, and we isolate the smallest case, when
it is not known whether this set has an open interior.

 \begin{remark}
    \label{rem:per42}
Choose an element $a\in k\smallsetminus\{0,\pm1\}$ and set $R=k[x_1,x_2,x_3,x_4]/I$,
where $I$ is the ideal generated by the following quadratic forms in four
variables:
  \[
x^2_1\,,\quad ax_1x_3+x_2x_3\,,\quad x_1x_4+x_2x_4\,,\quad x^2_2\,,
\quad x^2_3\,,\quad x_3x_4\,,\quad x^2_4\,.
  \]

Since $x_4$ is a Conca generator for $R$, Proposition \ref{prop:(e-1)p}
shows that the sets  $\LL_{1,3}(R)$ and $\LP^2_{1,3}(R)$ are equal,
open, and non-empty.  On the other hand, \cite[3.4]{GP} and Proposition
\ref{prop:(e-1)p} give, respectively, the strict inclusion and the
inequality below:
  \[
\LL_{2,6}(R)\supsetneq\LP_{2,6}^2(R)\ne\varnothing\,.
  \]

We do not know whether either set above has a non-empty interior in
$\mat 6{8}$.
 \end{remark}

\end{document}